\documentclass[11pt]{article}

\usepackage{amsfonts}
\usepackage{amsmath,amsthm,amscd,amssymb,mathrsfs,setspace}
\usepackage{latexsym, epsf, epsfig}
\usepackage{color}
\usepackage[hmargin=.95in,vmargin=.95in]{geometry}
\usepackage{graphicx}
\usepackage{hyperref}
\usepackage{cite}

\hypersetup{backref=true}
\newcommand{\ds}{\displaystyle}

\newcommand{\cA}{{\mathcal{A}}}

\newcommand{\cE}{{\mathcal{E}}}

\newcommand{\cB}{\mathcal{B}}

\usepackage{xspace,xypic}

\renewcommand{\vec}[1]{\mathbf{#1}}
\newcommand{\set}[1]{\left\{ {#1} \right\}}
\newcommand{\norm}[1]{\left\| {#1} \right\|}
\newcommand{\braket}[1]{\left\langle {#1} \right\rangle}
\newcommand{\BA}{\mathbb A}
\newcommand{\BB}{\mathbb B}

\theoremstyle{plain}
\newtheorem{theorem}{Theorem}[section]

\newtheorem{proposition}[theorem]{Proposition}
\newtheorem{corollary}[theorem]{Corollary}
\newtheorem{assumption}{Assumption}
\newtheorem{definition}{Definition}
\theoremstyle{remark}
\newtheorem{remark}{Remark}[section]
\numberwithin{equation}{section} \numberwithin{theorem}{section}
\numberwithin{remark}{section} 

\begin{document}

\title{Weak Solutions for a Poro-elastic Plate System}
 \author{\normalsize \begin{tabular}[t]{c@{\extracolsep{.8em}}c}
            { \large Elena Gurvich} &{ \large Justin T. Webster} \\
 \it University of Maryland, Baltimore County   \hskip.6cm  & \hskip.6cm \it University of Maryland, Baltimore County   \\
 \it Baltimore, MD &\it Baltimore, MD\\
    \it  gurv-3@umbc.edu &  \it websterj@umbc.edu \\
\end{tabular}}
\maketitle

\begin{abstract} {\noindent We consider a recent plate model obtained as a scaled limit of the three dimensional Biot system of poro-elasticity. The result is a ``2.5" dimensional linear system that couples traditional Euler-Bernoulli plate dynamics to a pressure equation in three dimensions, where  diffusion acts only transversely. We allow the permeability function to be time-dependent, making the problem non-autonomous and disqualifying much of the standard abstract theory. Weak solutions are defined in the so called quasi-static case, and the problem is framed abstractly as an implicit, degenerate evolution problem. Utilizing the theory for weak solutions for implicit evolution equations, we obtain existence of solutions. Uniqueness is obtained under additional hypotheses on the regularity of the permeability function. We address the inertial case in an appendix, by way of semigroup theory. The work here provides a baseline theory of weak solutions for the poro-elastic plate, and exposits a variety of interesting related models and associated analytical investigations.
  \\[.15cm]
\noindent {\bf Key terms}: poro-elasticity; plate; elliptic-parabolic; hyperbolic-parabolic; implicit evolution
 \\[.15cm]
}
\end{abstract}
\maketitle

\section{Introduction}
In this paper we address the weak solutions of a quasi-static, compressible Biot plate system. Our primary modeling reference is \cite{mikelic}, which itself names \cite{taber} as a motivating reference. The plate model presented in \cite{mikelic} is derived as a rigorous limit of the traditional 3D Biot equations \cite{show1,frenchpaper} (and references therein). In obtaining the 2D equations of motion through appropriate scalings, existence of weak solutions is naturally obtained  from those of the 3D Biot system. Such a poro-elastic plate model (or a suitable modification thereof) has been utilized recently, e.g., in \cite{bcmw} and \cite{rohan}, to capture the dynamics of a saturated porous plate interacting with other dynamics in a multi-layered structure. 

In the present analysis, we apply to the Biot plate the abstract framework which was developed for implicit, degenerate evolution equations in \cite{indiana} and discussed later in \cite{dbshow,showmono,showold}. (Of course this abstract theory may be applied to 3D Biot dynamics as well.) We permit the permeability function in the analysis to be time-dependent, resulting in a time-dependent principal operator for the pressure equation. A number of interesting features emerge in the study of this ``2.5" dimensional system, making a direct application of the available abstract theory nontrivial. Indeed, the regularity associated to plates (namely, higher order boundary conditions appearing in operator descriptions), as well as in-plane elliptic degeneracy in the pressure equation, create a variety of novel challenges for the analysis at hand. In addition, the issue of boundary conditions for the pressure, and their comparability to plate boundary conditions, are central issues in the work.

In what follows, we present a clear functional framework for the analysis of {\em weak solutions} to the quasi-static Biot plate system. To the best knowledge of the authors, the Biot plate system has not been treated in an operator-theoretic framework. We work to reduce the pressure-displacement system to an abstract, implicit evolution posed on appropriate spaces, for which we can invoke the theory of \cite[III.3]{showmono} (developed in \cite{indiana} and based on a classical variational result of Lions---Theorem \ref{Lions} below). We obtain weak solutions to the quasi-static problem in this time-dependent framework, enabling future work on various nonlinear modifications of the dynamics as they arise in relevant biological applications  \cite{bcmw,bgsw}. For self-containedness, we include the proofs of the main abstract theorems employed here, Theorems \ref{ImplicitExistence} and \ref{ImplicitUniqueness}, in Appendix A.

 We also consider the inertial case of the Biot plate system, which to our knowledge has also not been rigorously addressed in the literature. We relegate the discussion of this case to Appendix B, owing to the fact that the approach is thematically distinct; indeed, for the inertial system we utilize traditional semigroup theory, as motivated by \cite{thermo,redbook}. We point out that the analysis is complicated by the subtle nature of the regularity classes in the problem, themselves associated to the disparate nature of the in-plane and transverse coupling. For the inertial, compressible case, we obtain generation of a strongly continuous semigroup of contractions on the appropriate state space \cite{pazy,redbook}.
 
 In view of space considerations and the mathematical emphasis of the work at hand, we omit an in depth discussion of the traditional poro-elasticity. We provide some references for the applications of  Biot models to physical systems from biological applications \cite{bcmw,bgsw},  to geosciences \cite{show1} (and references therein), to engineering \cite{rohan,taber} (and references therein). On the mathematical side, early works on the Biot system include \cite{zenisek,frenchpaper,indiana}. The work of Showalter (see \cite{show1}) largely developed the theory in the context of implicit, degenerate equations \cite{showmono,showold,dbshow,indiana}. {See \cite{referee} for an FEM discussion of the approximation of weak solutions to relevant linear plate models.} The coupled dynamics considered here are challenging dynamics to work with, both from the point of view of variational approaches to weak solutions, as well as refinements of traditional semigroup methods.
\vskip.2cm
\noindent	{\bf Summary and Goals}: We place this Biot plate problem, as derived in \cite{mikelic}, in the context of {\em weak solutions for implicit, degenerate problems} as presented in \cite[III.3]{showmono}, and thereby obtain weak well-posedness of the system. The inertial, compressible case, being thematically disparate (it is {\em not an implicit problem}), is entirely contained in Appendix B, where we show that the associated evolution operator generates a strongly continuous semigroup of contractions on an appropriate state space.
	
\subsection{Physical Configuration and Mathematical Models}\label{modelsec}
In what follows, denote $\mathbf x = \langle x_1, x_2, x_3\rangle \in \mathbb R^3.$ We denote differentiation with respect to $x_i$ by $\partial_i$.

	\begin{figure}[h]
        \centering
        \includegraphics[width=0.5\textwidth]{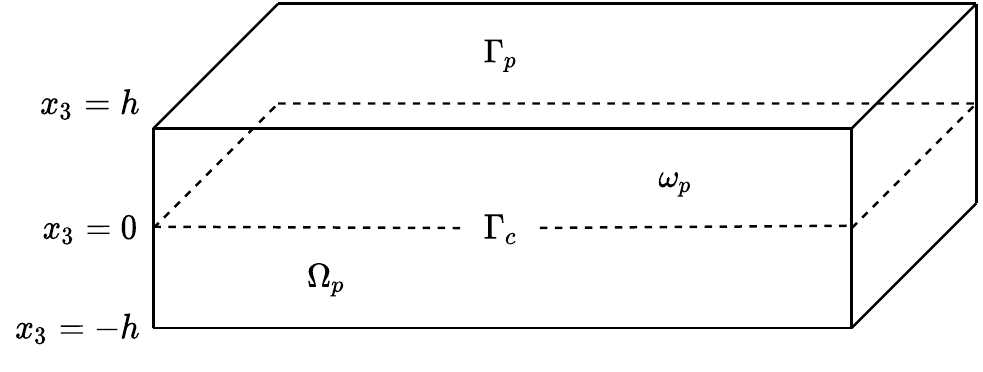}
    \end{figure}
    We consider a thin poro-elastic region $\Omega_p = \omega_p \times (-h,h)$ with center surface $\omega_p = (0,1)\times(0,1)$ at $x_3 = 0$. The corresponding boundaries are denoted $\Gamma_p = \partial \Omega_p$ and $\Gamma_c = \partial \omega_p$. We consider $x_3 \in (-h,h)$ as the transverse variable, whereas $x_1$ and $x_2$ denote the longitudinal and lateral directions, respectively. The model below was recently derived from 3D Biot dynamics in \cite{mikelic}, and the corresponding in- and out-of-plane dynamics are shown to decouple.  Therefore, in this treatment, we only consider the transverse dynamics, as is often done in standard linear plate theory. Let $w(x_1,x_2,t)$ denote the transverse displacement of the center surface ($x_3 = 0$) of the plate, with the Kirchhoff-Love hypothesis in place\footnote{among other assumptions, that plate filaments remain orthogonal to the deflected center surface \cite{ciarlet2}}. Let $p(\vec{x},t)$ represent the fluid pressure defined on the entire region $\Omega_p$. From \cite{mikelic}, the physical quantities of interest are: 
    \begin{itemize}
    \item $\rho_p \ge 0$ is the plate's density or {\em inertial parameter}; when $\rho_p>0$, the system is said to be {\em inertial}; when $\rho_p=0$, the system is said to be {\em quasi-static};
    \item $D>0$ is taken to be the plate's {\em flexural rigidity}, scaled to the system at hand;
    \item $\alpha>0$ is the so-called {\em Biot-Willis constant}, again appropriate scaling;
    \item $c_p\ge 0$ is the {\em storage coefficient} of the fluid-solid matrix, corresponding to net compressibility of constituents;
    \item $k_p \ge 0$ is the permeability of the porous matrix comprising the plate.
    \end{itemize} 
    The physical scalings for the above parameters and an extensive discussion are given in \cite{mikelic} (and, for general Biot systems, see \cite{show1} and references therein). A quantity of interest is the so called {\em fluid content} of the system, which is given here by ~$\zeta = c_pp-x_3\alpha \Delta w$.
    We allow the permeability to be a general function of space and time $k_p = k_p(\mathbf x,t)$, in order to permit recent applications where the fluid content's effect on the porosity (and hence permeability) is relevant---see, e.g., \cite{bgsw}. 
    
    The Biot dynamics of an inertial poro-elastic plate \cite{mikelic} are governed by
	\begin{equation}\label{FDPlate}
	\begin{cases}
	    \rho_p w_{tt} + D \Delta^2 w + \alpha \Delta \int_{-h}^h x_3p~dx_3 = f &\text{in } \omega_p,\\
	    [c_p p - \alpha x_3 \Delta w]_t - \partial_3(k_p\partial_3 p) = g & \text{in } \Omega_p,\\
	    w(x_1,x_2,0)=w_0(x_1,x_2),~~w_t(x_1,x_2,0)=w_1(x_1,x_2) & \text{in } \omega_p,\\
	    [c_pp-\alpha x_3\Delta w](\mathbf x,0)=d_0(\mathbf x) & \text{in } \Omega_p.
	\end{cases}
	\end{equation}
\begin{remark} In the reference \cite{mikelic}, the equations above are presented in non-inertial form with $\rho_p=0$. As discussed above, the in-plane dynamics are considered as there well. The in-plane equations of motion constitute (in the inertial case) a 2D dynamic (Navier) system of elasticity in $(x_1,x_2)$. 
In the scaling analysis presented in this reference, the transverse dynamics and in-plane dynamics formally decouple. For this reason we---and other authors \cite{bcmw}---consider only transverse dynamics, although this point may be relevant to the discussion of the compressible case $c_p=0$ (see Section \ref{future}).\end{remark}

One can rescale the system (\ref{FDPlate}) in a convenient form. In the pressure terms, take $\vec{x} \mapsto \rho_p c_p^{-1}\vec{x}$, and consider the state $v = w_t$. Then the dynamics become (with bars representing rescaled constants):
    \begin{equation}\label{FDPlate2}
	\begin{cases}
	    w_t-v = 0 & \text{in } \omega_p\\
	    v_t + \overline{D} \Delta^2 w + \overline{\alpha} \Delta \int_{-h}^h x_3p\,dx_3 = f &\text{in } \omega_p,\\
	    \partial_tp - \overline{\alpha}_p x_3\Delta v - \partial_3(\overline{k}_p\partial_3 p) = g & \text{in } \Omega_p,\\
	    	    w(x_1,x_2,0)=w_0(x_1,x_2),~~w_t(x_1,x_2,0)=w_1(x_1,x_2) & \text{in } \omega_p,\\
	    p(\mathbf x,0)-\alpha x_3\Delta w_0(x_1,x_2)=d_0(\mathbf x).
	\end{cases}
	\end{equation}
	 In the analysis of the {\it inertial} Biot plate system (Appendix B),  we will consider the rescaled equation (omitting bars). We will consider the inertial $\rho_p>0$ compressible $c_p > 0$ case, which is of {\em hyperbolic-parabolic} type, and is closely related in structure to systems of thermoelasticity \cite{thermo,redbook} though with certain peculiarities here.
	 	\begin{remark} When $\rho_p>0$ and $c_p=0$, the result is an implicit {\em hyperbolic} problem. This will be the focus of a future work, as, to the knowledge of the authors, there does not appear to be much discussion in the literature.\end{remark}

	The  main focus of the work at hand is the so called {\em quasi-static} version of \eqref{FDPlate}. When the deformation of the elastic matrix is slow and inertia is negligible, we then take $\rho_p = 0$ to obtain
	\begin{equation}\label{QSPlate}
	\begin{cases}
	    D \Delta^2 w + \alpha \Delta \int_{-h}^h x_3p\,dx_3 = f &\text{in } \omega_p,\\
	    [c_p p - \alpha x_3 \Delta w]_t - \partial_3(k_p\partial_3 p) = g & \text{in } \Omega_p,\\
	   	[c_pp-\alpha x_3\Delta w](\mathbf x,0)=d_0(\mathbf x) & \text{in } \Omega_p.
	\end{cases}
	\end{equation}
We note that the initial conditions have requisitely changed, reflecting the loss of the independently specified plate velocity state. The above system is of {\em elliptic-parabolic} type, and can be written as an implicit evolution equation \cite{showmono,dbshow}; when the constituents are incompressible $c_p=0$ (not considered in this treatment), the system may also degenerate.
	
	By way of boundary conditions, we are inclined toward  Neumann-type conditions for the pressure on the top and bottom $\{x_3=\pm h\}$ parts of $\Gamma_p$, and hinged boundary conditions for the rectangular plate on $\Gamma_c$. {\em The pressure equation itself requires no boundary condition on the lateral walls, as there is no $(x_1,x_2)$-differential action explicitly acting on $p$}. On the other hand, the elasticity equation \eqref{QSPlate}$_1$ compares $\Delta w$ and the $x_3$-moment of $p$ under a Laplacian, which suggests a boundary condition thereon.	Thus we have:
		\begin{equation}\label{BC}
        \begin{cases}
               w  = 0;~~D\Delta w+\alpha\int_{-h}^hx_3 p~dx_3=0 &\text{on } \Gamma_c,\\
            k_p\partial_{\vec{n}} p = 0 &\text{on } \{x_3=h\} \cup \{x_3=-h\}.
        \end{cases}
	\end{equation}
	Here, $\vec{n}$ is the unit normal vector on the boundary $\Gamma_p$, and on $\{x_3 = \pm h\}$ we have $\partial_{\mathbf n} = \pm \partial_3$. Note that on $\Gamma_c = \partial \omega_p$ the normal $\mathbf n$ coincides with $\Gamma_p$.
	
	\begin{remark} Our baseline interest in Neumann boundary conditions on the top and bottom of the domain comes from a desire to couple plate dynamics with, e.g., Stokes or 3D Biot systems---see \cite{bcmw}. In Section \ref{configs} below, we discuss various other configurations and related forthcoming work. We suffice to say that there is a nontrivial interaction between the choice of these boundary conditions and a direct application of the abstract theory. 
	\end{remark}
	\begin{remark}
	We also point out that the higher order pressure/plate boundary condition appears at the level above that of weak solutions. On the other hand, it becomes necessary to consider such boundary conditions in Appendix B, where  semigroup theory is invoked and strong solutions are obtained. 
	 \end{remark}

	\subsection{Natural Spaces}\label{spaces} We work in the $L^2(U)$ framework here, where $U \subseteq \mathbb R^n$ for $n=2,3$ is a given spatial domain. We denote standard $L^2(U)$ inner products by $(\cdot,\cdot)_U$.  Traditional Sobolev spaces of the form $H^s(U)$ and $H^s_0(U)$ (along with their duals) will be defined in the standard way \cite{kesavan}, with the $H^s(U)$ norm denoted by $||\cdot||_{s,U}$. For a Banach space $Y$ we denote its dual as $Y'$ and denote the associated duality pairing as $\langle \cdot, \cdot\rangle_{Y'\times Y}$.
	 
		Define the anisotropic Sobolev spaces and associated norm
    \begin{align*}
        H^{0,0,k}(\Omega_p) &= \set{ \phi \in L^2(\Omega_p): \partial_3 \phi, ..., \partial_3^k \phi \in L^2(\Omega_p)}, \quad k \geq 1,\\
        \norm{\phi}_{0,0,k} &= \left(\sum_{i=0}^k \norm{\partial_3^i \phi}_{0,\Omega_p}^2\right)^{1/2}.
    \end{align*}
    We will give special status to $V \equiv H^{0,0,1}(\Omega_p)$ in the analysis of the pressure equation. The mappings $\gamma_k \in \mathscr L(H^{k+1}(\Omega_p),H^{k+1/2}(\Gamma_p))$ denote the standard trace maps \cite{kesavan}, $k=0$ corresponding to the the Dirichlet, and $k=1$ to the Neumann trace.

    Lastly, define the space $W = H^2(\omega_p) \cap H_0^1(\omega_p)$, taken with the Laplacian norm, i.e., $||w||_W \equiv ||\Delta w||_{L^2(\omega_p)}$.  Note that the space $W$ is naturally associated to the (square root of the) hinged biharmonic operator, and this norm is equivalent to the standard $H^2(\omega_p)$ norm on $W$.

	\subsection{Main Result}
We begin by stating the definition of a weak solution in the quasi-static case. First, let us consider some the $d_0 \in V'$, with $f \in L^2(0,T;W')$ and $g \in L^2(0,T;V')$.

\begin{definition}[Quasi-static Weak Solution]\label{weaksol}
A solution to \eqref{QSPlate} with {$c_p> 0$} is represented by a pair of functions 
$$w \in L^2(0,T;W) \ \ \text{and}\ \  p \in L^2(0,T;V),$$ 
with $\zeta = c_pp -x_3\alpha \Delta w \in L^2(0,T;L^2(\Omega_p)) \cap H^1(0,T;V')$, such that: 
\vskip.1cm
\noindent (a) the following variational forms are satisfied for any $z \in L^2(0,T;W)$, and any \\$q \in \{ w \in L^2(0,T;V)\cap H^1(0,T;L^2(\Omega_p))~:~w(T)=0\}$:
\begin{align}  \label{weakform1}
D \int_0^T(\Delta w, \Delta z)_{\omega_p}dt +& \alpha  \int_0^T(p, x_3 \Delta z)_{\Omega_p}dt = \int_0^T\langle f, z \rangle_{W'\times W}dt,\\
\label{weakform2}
\int_0^T \big(k_p\partial_{x_3} p, \partial_{x_3} q\big)_{\Omega_p}& ~dt-\int_0^T(\zeta, q_t)_{\Omega_p}~dt = \int_0^T \langle g,q \rangle_{V'\times V} ~dt
\end{align}
(b) the initial condition holds in the sense of $V'$, namely that ~~$\ds \lim_{t\searrow 0} \zeta(t) = d_0 \in V'.$
\end{definition}

We now state two assumptions on the time-dependent permeability that will be crucial to our main results here.
	\begin{assumption}\label{KAssumption} [Bounds for Existence] We assume that permeability $k_p$ is an $L^\infty(\Omega_p \times (0,T))$ function such that 
    \[
        0 < k_* \leq k_p(\vec{x},t) \le k^*, \quad \forall\, \vec{x} \in \Omega_p, ~ \forall\, t \in [0,T].
    \]
    \end{assumption}
     To show uniqueness, we will require an additional assumption on the first derivative of $k_p$ which will ensure that $\Big\{-\partial_3[k_p(t)\partial_3]~:~t \in [0,T]\Big\}$ constitutes a {\em regular family} \cite{showmono}.
    \begin{assumption}\label{KAssumption2} [Bounds for Uniqueness]
    Assume that for a.e $\mathbf x \in \Omega_p$ the function $k_p(\vec{x},\cdot)$ is absolutely continuous in $t \in [0,T]$, with $| \partial_tk_p(\vec{x},t)| \leq K(t)$ and $K \in L^1(0,T)$.
	\end{assumption}

	The main result in this paper is the well-posedness of weak solutions the Biot plate in (\ref{QSPlate}). 
	\begin{theorem}[Quasi-static Weak Solution]\label{main2}
	    Let $c_p>0$, and $d_0 \in L^2(\Omega_p)$, with $f \in H^1(0,T;W')$, $g \in L^2(0,T;V')$. Assume that permeability $k_p$ satisfies the hypotheses of Assumption \ref{KAssumption}, then there exists a  weak solution $(w,p) \in L^2(0,T; V) \times L^2(0,T;W)$ to (\ref{QSPlate}) in the sense of the definition above. This solution satisfies the stability estimate
	          \begin{equation}\label{stability**}
       ||p||_{L^2(0,T;V)}^2+||w||_{L^2(0,T;W)}^2 \le C\left[||f||_{H^1(0,T;W')}^2+||g||^2_{L^2(0,T;V')}+||d_0||^2_{L^2(\Omega_p)}\right].
        \end{equation}
 Moreover, if $k_p$ satisfies Assumption \ref{KAssumption2}, then the solution is unique. 	\end{theorem}
	
	\begin{remark} The type of initial conditions taken (i.e., $w(0)$ and/or $p(0)$ and/or $d_0$) is a non-trivial issue, as is the space in which they are taken. It suffices to say that in the case that $c_p>0$, specifying $d_0 \in L^2(\Omega_p)$ will be equivalent to specifying $p(0) \in L^2(\Omega_p)$, and subsequently, $w(0) \in W$. In other situations, this may not be the case---see the discussion in \cite{bw,borisnew}. Note, we (as well as nearly all other authors cited here) require higher regularity of initial data than is required by the baseline definition of the weak solution. \end{remark}
	
	Finally, we recall that the entire discussion of the inertial, compressible plate has been bundled in a self-contained appendix (Appendix B). There we show that the evolution operator associated to the $\rho_p, c_p>0$ initial boundary value problem generates a traditional $C_0$ semigroup (in the sense of \cite{pazy}), rather than an implicit semigroup (in the sense of \cite{show1,frenchpaper}).

\section{Implicit, Degenerate Equations}\label{abstractsec}
The traditional, linear Biot system in 3D can be placed into the framework of an implicit, degenerate evolution equation \cite{show1,bw} in the pressure variable $p$. We demonstrate the applicability of a similar approach here (in the quasi-static case). Let us now consider a general abstract system (with time-dependent coefficients) written in the variable $p$ and operators $\mathcal B$ and $\{\mathcal A(t)~:~t \in [0,T]\}$:
\begin{equation}\label{gensys}\begin{cases}
[\mathcal B p]_t+\mathcal A(t) p = S \\
[\mathcal Bp](0) = d_0.\end{cases}\end{equation}
The system is {\em implicit} since the time derivative does not appear directly on $p$---in fact $p$ need not differentiable at all for solutions---and the natural initial data is $\mathcal [Bp](0)$. The system is {\em degenerate} if $\mathcal B$ has a non-trivial kernel in some sense (depending on the interpretation of the equation). There are many different applications for such systems \cite{showmono,dbshow}; for Biot problems posed on $L^2$, $\cB$ is a ``nice" zeroth order, nonlocal operator, and $\cA(t)$ is a maximal monotone operator on for each $t$.
There is a wealth of literature on implicit evolution equations, see, for instance, \cite{dbshow,showmono,showold,frenchpaper,indiana}.
	
	\subsection{Theory of Implicit, Degenerate Equations: Time-dependent Case}\label{abstractsec*}
	In this section we provide the general theory for \eqref{gensys}. What is presented below is not even the most general framework available as presented in \cite{showmono} or \cite{dbshow}. We outline the main theorems and provide the broad context for what we will employ below. For self-containedness---and since this theory does not appear prevalently in the literature---we provide the proofs of these theorems in Appendix A. 
		\subsubsection{Cauchy Problem and Definition of Solutions}\label{CauchySection}
		    
		    Let $V$ be a separable Hilbert space with dual $V'$ (which is not identified with $V$ here). Assume $V$ densely and continuously includes into another Hilbert space $H$, which {\em is} identified with its dual:
            $$V \hookrightarrow H \equiv H' \hookrightarrow V'.$$
            We denote the inner-product in $H$ simply as $(\cdot,\cdot)$, with $(h,h)=||h||_H^2$ for each $h \in H$. Similarly, we denote the $V'\times V$ duality pairing as ~$\langle \cdot, \cdot \rangle$. (When $h\in H$, we can identify $\langle h, h \rangle = ||h||_H^2$ as well.)
            Assume for each $t \in [0,T]$ that $\cA(t) \in \mathscr L(V,V')$  with $\braket{\mathcal A(\cdot)u,v} \in L^{\infty}(0,T)$ for each pair $u,v \in V$. Take $\cB \in \mathscr L(H)$. 
            Finally, suppose that $u_0 \in H$ and $S \in L^2(0,T;V')$ are the specified data. 
            
            In this setup, we can define the weak (implicit-degenerate) Cauchy problem to be solved as:
            
    \noindent         Find $u \in L^2(0,T;V)$ such that
            \begin{equation}\label{ImplicitCauchy}
            \begin{cases}
            \dfrac{d}{dt}[\cB u]+\cA(\cdot)u=S \in ~L^2(0,T;V') \\
            [\cB u](0)=\cB u_0 \in V'.
            \end{cases}
            \end{equation}
            The time derivative above is taken in the sense of $\mathscr D'(0,T)$, and since such a solution would have $\cB u \in H^1(0,T;V')$ (with the natural inclusion $V \hookrightarrow V'$ holding), $\cB u$ has point-wise (in time) values into $V'$ and the initial conditions makes sense through the boundedness of $\cB$.             
            \begin{remark}\label{ICremark} We note that it would, in fact, be natural to consider initial data specified only as $[\cB u](0)=d_0$, rather than specifying a $u_0 \in H$ such that $d_0 = \cB u_0$. This is a notably subtle issue. When $\cB$ is invertible on the space where $d_0$ is specified, then the aforementioned assumptions are equivalent. On the other hand, in much of the literature (e.g., \cite{bw,bcmw,bgsw,showold, showmono, zenisek}), the data is specified on $u_0$. See \cite{show1} in the {\em holomorphic} cases for data specified only as $d_0$.\end{remark}
            
            We also define a weak solution to the implicit Cauchy problem in \eqref{ImplicitCauchy} as follows: 
            \begin{definition}[Weak Solution]\label{weakimplicitsol} The function $u \in L^2(0,T;V)$ with
            \begin{equation}\label{weakweak}
                -\int_0^T(\cB u(t),v'(t))dt+\int_0^T\langle \cA(t)u(t), v(t)\rangle dt=\int_0^T\langle S(t), v(t) \rangle dt+(\cB u_0,v(0))
            \end{equation}
            holding for all 
            $$v \in \{ w \in L^2(0,T;V)\cap H^1(0,T;H)~:~w(T)=0\}$$
            is called the {\em weak solution} to \eqref{ImplicitCauchy}.
     \end{definition}
            This definition is equivalent (via the standard mechanism \cite[p.115]{showmono}) to a function $u \in L^2(0,T;V)$ that satisfies the Cauchy problem \eqref{ImplicitCauchy}  as described above.
						
		\subsubsection{Existence and Stability}\label{ext}
		The theorem we present for existence of solutions in the above framework seems to be the most general available. We note that the result is, indeed, quite general, and seems to unique in the literature in that it accommodates the time-dependent setting. The original reference is \cite{indiana}, but a nice exposition by the same author can be found in \cite[III.3]{showmono}. There are other approaches for the autonomous problem that can be seen in, e.g., \cite{show1,showold,showmono,zenisek,dbshow}.
        
        From \cite{showmono}, quite general criteria  yield existence. Note that we define monotone to mean:
        $$\text{ $\cA(t)$ is monotone on $V$ } ~\iff~\langle \cA(t)p,p\rangle \ge 0,~~\forall ~p \in V, ~ t \in [0,T].$$
        In the case of $\cB$, which is maps $H$ to $H$, the definition is simplified to mean
        $$\text{ $\cB$ is monotone on $H$ } ~\iff~(\cB p,p) \ge 0,~~\forall ~p \in H.$$
        
        The proof of the result below hinges on a classical generalization of Lax-Milgram \cite[p.109]{showmono}.
        \begin{theorem}\label{ImplicitExistence}
        Suppose that $V,H$ are separable Hilbert spaces, with $V\hookrightarrow H \hookrightarrow V'$ as dense, continuous inclusions. Also, let $\cA(t)$ be monotone as above for each $t \in [0,T]$, and assume $\cB$ is self-adjoint and monotone on $H$. Let $u_0 \in H$ and $S \in L^2(0,T;V')$. 
        
        If there exist constants $\lambda, c>0$ such that
        $$2\langle \cA(t)v, v\rangle+\lambda (\cB v,v) \ge c||v||_V^2,~~\forall v \in V,~~\forall t \in [0,T],$$
        then there exists a weak solution to the Cauchy problem described in \eqref{weakweak}.
        The particular solution obtained satisfies the stability estimate
        \begin{equation}\label{stability}
        ||u||_{L^2(0,T;V)}^2 \le C(c,\lambda)\left[||S||_{L^2(0,T;V')}^2+(\cB u_0,u_0)\right].
        \end{equation}
        \end{theorem}
        It is important to note that only the particular solution given in the construction satisfies the stability-type estimate. On the other hand, to quote \cite{showmono}, ``much non-uniqueness is possible" at this abstract level---owing to the ability to translate the solution by $ker(\cB)\cap ker(\mathcal A(t))$. Additionally, note that the above theorem requires specification of $u(0)$, rather than the natural quantity $[\mathcal B p](0)$; this is a point of note in the discussion of implicit evolution equations \cite{borisnew}. 
        
        \subsubsection{Uniqueness}
        The general theorem for uniqueness in this context utilizes the notion of a {\em regular family} of operators. 
        \begin{definition}\label{regular}
            The family $\{\cA(t):V \to V'~:~t \in [0,T]\}$ of operators is regular if for every pair $u,v \in V$, the function $\langle \cA(\cdot)u,v\rangle $ is absolutely continuous on $[0,T]$ and there exists a $K \in L^1(0,T)$ such that
            \[
                \left| \frac{d}{dt} \langle \cA(t)u,v\rangle\right| \leq K(t) \|u\|_V~\|v\|_V, \quad u,v \in V, \text{ a.e. }t \in [0,T].
            \]
        \end{definition}
        \noindent Note that the proof of the following theorem also relies on the Lions result \cite[p.109]{showmono}.
        \begin{theorem}\label{ImplicitUniqueness}
            With the hypotheses of Theorem \ref{ImplicitExistence} in force, assume that $\{\cA(t)~:~t \in [0,T]\}$ is a regular family of self-adjoint operators. Then the solution in Theorem \ref{ImplicitExistence} is unique and it requisitely satisfies stability estimate \eqref{stability}.             
        \end{theorem}
\section{Plate Model and Abstract Setup}
We consider the quasi-static ($\rho_p=0$) compressible ($c_p>0$) case of the Biot plate dynamics here.
	\subsection{PDE Model}
As described in Section \ref{modelsec}, the primary system of interest here is:
	\begin{equation}\label{QSPlate*} 
    	\begin{cases}
    	    D \Delta^2 w + \alpha \Delta \int_{-h}^h x_3p\,dx_3 = f &\text{in } \omega_p,\\
    	    [c_p p - \alpha x_3 \Delta w]_t - \partial_3(k_p\partial_3 p) = g & \text{in } \Omega_p,\\
    	   	[c_pp-\alpha x_3\Delta w](\mathbf x,0)=d_0(\mathbf x) & \text{in } \Omega_p,\\
            w  = 0;~~D\Delta w+\alpha\int_{-h}^hx_3 p~dx_3=0 &\text{on } \Gamma_c,\\
            \partial_{3} p = 0 &\text{on } \{x_3=\pm h\}.
        \end{cases}
	\end{equation}
Note that for $x_3=\pm h$ we have simplified the Neumann condition using the Assumption \ref{KAssumption} (so $k_p>k_*>0$) and that $\mathbf n = \pm \mathbf e_3$ on the top and bottom of $\Gamma_p$.
	
	\subsection{Operators}\label{Operators}
	
    The goal of this section is to establish the specific operators that will allow us to place the dynamics in the framework of the abstract Theorems \ref{ImplicitExistence} and \ref{ImplicitUniqueness}. Please recall the definitions  in Section \ref{spaces}.

    \subsubsection{Constituent Operators}\label{constOp}
    
    Define, for $p \in L^2(\Omega_p)$, the moment operator $\mathcal{K} : L^2(\Omega_p) \to L^2(\omega_p)$ by ~~$\ds \mathcal{K}(p) = \int_{-h}^h x_3 p\,dx_3.$ Similarly, let $\tilde{\mathcal{K}} \in \mathscr{L}(L^2(\omega_p),L^2(\Omega_p))$ be defined by $\tilde{\mathcal{K}}(q) = x_3q$ for every $q \in L^2(\omega_p)$. This is suggestive of the fact that
    \begin{equation}\label{KAdjoint}
        (\mathcal{K}p,q)_{\omega_p} = (p, \tilde{\mathcal K}q)_{\Omega_p}, \qquad p \in L^2(\Omega_p), q \in L^2(\omega_p).
    \end{equation}

    \begin{proposition}\label{Kbounded}
	    For all $s \geq 0$, $\mathcal{K} \in \mathscr{L}(H^s(\Omega_p), H^s(\omega_p))$ and $\tilde{\mathcal{K}} \in \mathscr{L}(H^s(\omega_p),H^s(\Omega_p))$. 
	\end{proposition}
	\begin{proof}
	    Let $s \geq 0$ and $p \in H^s(\Omega_p)$ be arbitrary. Recall that
        \[
            \norm{\mathcal{K}(p)}_{s,\omega_p}^2 \equiv \sum_{|\alpha| \leq s} \norm{D_{\omega_p}^\alpha \mathcal{K}(p)}_{0,\omega_p}^2,
        \]
        where $D^\alpha_{\omega_p} = \partial_1^{\alpha_1}\partial_2^{\alpha_2}$. So, for any multi-index $|\alpha| \leq s$,
        \begin{align} \label{Kseminorm}
            \norm{D_{\omega_p}^\alpha \mathcal{K}(p)}_{0,\omega_p}^2 &
            = \int_{\omega_p} \left( \int_{-h}^h x_3 D_{\omega_p}^\alpha p\,dx_3\right)^2 \, d\omega_p = \int_{\omega_p} |\mathcal{K}(D_{\omega_p}^\alpha p)|^2 \, d\omega_p.
        \end{align}
        We have via Cauchy-Schwarz that
        \begin{align}\label{Kabs}
            |\mathcal{K}(D_{\omega_p}^\alpha p)| &= \left| \int_{-h}^h x_3 D_{\omega_p}^\alpha p\, dx_3\right| 
            \le \sqrt{\frac{2h^3}{3}}\left( \int_{-h}^h |D_{\omega_p}^\alpha p|^2\,dx_3\right)^{1/2}
        \end{align}
        for every multi-index $|\alpha| \leq s$. It follows from (\ref{Kseminorm})--(\ref{Kabs}) that
        \begin{equation}\label{Knorm}
            \norm{\mathcal{K}(p)}_{s,\omega_p}^2 = \sum_{|\alpha| \leq s} \norm{D_{\omega_p}^\alpha \mathcal{K}(p)}_{0,\omega_p}^2 \leq \frac{2h^3}{3} \norm{p}_{s,\Omega_p}^2 \quad \forall\, p \in H^s(\Omega_p), s \geq 0.
        \end{equation}

    The boundedness asserted for $\tilde{\mathcal K}$ is clear.
    
	\end{proof}
The following is an immediate consequences of (\ref{KAdjoint}) and Proposition \ref{Kbounded}.
\begin{corollary} The operator
$\mathcal K :  L^2(\Omega_p) \to L^2(\omega_p)$ has adjoint $\mathcal K^*=\tilde{\mathcal K}: L^2(\omega_p) \to L^2(\Omega_p)$.
\end{corollary}

	Now, for each $t \in [0,T]$, let $A(t) : H^{0,0,1}(\Omega_p) \to H^{0,0,1}(\Omega_p)' $ be defined through the bilinear form
    \[
        A(p,q;t) = (k_p \partial_3 p, \partial_3 q)_{\Omega_p},\quad \forall\, p,q \in H^{0,0,1}(\Omega_p).
    \]
Thus we can identify a realization of $A(t)$  for each $p \in H^{0,0,1}(\Omega_p)$, $$p \mapsto A(t)p\equiv A(p,[\cdot];t) \in H^{0,0,1}(\Omega_p)'.$$
Note that when Assumption \ref{KAssumption} is in force, $A(t)$ does not degenerate since $0<k_* \le k_p$. 

When $k$ is smooth in space, we can identify $A(t)$ with the unbounded operator having  $$\mathcal D(A(t)) \equiv \Big\{p \in H^{0,0,2}(\Omega_p)~:~\partial_3p=0~\text{ for } x_3=\pm h\Big\}$$ and differential action 
$$        A(t)p = -\partial_3 [k_p(\vec{x},t)\partial_3 p],\quad \forall\, p \in \mathcal D(A(t)).$$ In this case
$A(t) : \mathcal D(A(t))\subset L^2(\Omega_p) \to L^2(\Omega_p)$ is self-adjoint for each $t \in [0,T]$.

The following proposition is immediate.
\begin{proposition}
The operator $A(t): V \to V'$, interpreted as above, is monotone in the sense of  Section \ref{ext}.
\end{proposition}
   
       We similarly define the hinged biharmonic operator $\mathcal{E} : L^2(\omega_p) \to L^2(\omega_p)$ with domain $\mathcal D(\cE) \equiv \{ w \in H^4(\omega_p)\cap H_0^1(\omega_p) ~:~\Delta w\big|_{\Gamma_c}=0\}$ and action $$\mathcal{E}w = \Delta^2 w ~\text{for} ~w \in \mathcal{D}(\cE).$$ The unbounded operator above is strictly positive, self-adjoint, with compact inverse \cite{ciarletbook}. As such, the  positive square root $\cE^{1/2}:W \to L^2(\Omega)$ is defined and self-adjoint \cite{ciarletbook,redbook}. Indeed, denoting the standard Dirichlet Laplacian $[-\Delta_D] : H^2(\omega_p)\cap H_0^1(\omega_p)\to L^2(\omega_p)$, we can identify $\cE^{1/2}=-\Delta_D$ as operators (thus identifying $\mathcal D(\cE^{1/2}) = W$).
      As above, we can associate $\cE$ to the bilinear form
    \[
        e(w,u) = (\Delta w, \Delta u)_{\omega_p},\quad \forall\, w,u \in W,
    \]
    and we can give an interpretation to $\cE: W \to W'$ through the bilinear form, and it is an isomorphism in this context. With $\omega_p$ a rectangle, and hinged boundary conditions in force, full elliptic regularity for the biharmonic operator $\cE$ is available \cite{redbook,ciarlet2}. Thus, $\cE: \mathcal D(\cE) \to L^2(\omega_p)$ is an isomorphism as well. 
    
We can similarly interpret $[-\Delta_D]: H_0^1(\omega_p) \to H^{-1}(\omega_p)$ through the associated standard bilinear form. Moreover, by the density of the continuous inclusion ~$H^{2}(\omega_p)\cap H_0^1(\omega_p)  \hookrightarrow H_0^1(\omega_p)$, we have that $$H^{-1}(\omega_p) \hookrightarrow [H^2(\omega_p)\cap H_0^1(\omega_p)]'=[\mathcal D(\Delta_D)]'$$ as another dense, continuous inclusion \cite{ciarletbook}.

	\subsubsection{Pressure-to-Laplacian Operator}

	The key insight in studying the quasistatic 3D Biot system is the introduction of a ``pressure-to-divergence" operator \cite{frenchpaper,show1,cao2,bw}, allowing one to rewrite the system as a single implicit equation. Here, we follow suit and introduce an operator to reduce the plate system to an implicit evolution problem.   We define the composite operator $B : L^2(\Omega_p) \to L^2(\Omega_p)$ by the chain:
	    \begin{equation}\label{diagram}
        \xymatrixcolsep{3.5pc}\xymatrix{
            L^2(\Omega_p) \ar[r]^-{\mathcal{K}}  & L^2(\omega_p) \ar[r]^-{-\alpha\Delta_D} & 
            W' \ar[r]^-{(D\mathcal{E})^{-1}} & W \ar[r]^-{\Delta_D} & L^2(\omega_p) \ar[r]^-{-\alpha\tilde{\mathcal{K}~}}  & L^2(\Omega_p).
        }
    \end{equation}
	For the sake of brevity, we denote the coefficient $\beta = \alpha^2/D$. In this case we see that
    \[
        Bp= \beta\tilde{\mathcal{K}}\Delta_D\mathcal{E}^{-1}\Delta_D\mathcal{K}(p) = \beta\tilde{ \mathcal K} \mathcal K(p).
    \]
    Note that $\beta$ is a positive constant in accordance to the physical quantities defined in Section \ref{modelsec}.
    It is clear, then that $Bp = -\alpha \tilde{\mathcal K} \Delta_D w$ when $\mathcal Ew = -\frac{\alpha}{D}\Delta_D\mathcal K(p)$.
    
    \begin{remark}
    While it may seem excessive to define the operator in this way, we point out that, in general, the structure of the $B$ operator should follow the diagram in \eqref{diagram}. In this particular instance, our choice of boundary conditions/configuration allows a dramatic simplification which allows us to nicely illustrate the theory of implicit equations.
    \end{remark}
    We note that the action of $B$ (as above) extends immediately to $V$ (since $V$ encodes only $L^2(\Omega_p)$ information with respect to $(x_1,x_2)$). 
       We now consider the boundedness properties of the $B$ operator. 
    \begin{proposition}\label{prop1} $B \in \mathscr L(L^2(\Omega_p))$. 
    \end{proposition}
        That $B \in \mathscr{L}(L^2(\Omega_p))$ follows directly from Proposition \ref{Kbounded}.  
        
     \begin{remark} In future considerations of smooth solutions, we can consider an auxiliary  space
    ~~$U \equiv \Big\{p \in L^2(\Omega_p)~:~\mathcal Kp \in \mathcal D(\Delta_D)\Big\}.$
    The action of $B$ extends to the space $U$ as follows: \vskip.1cm
    	    \begin{equation}\label{diagram*}
        \xymatrixcolsep{3.5pc}\xymatrix{
            U \ar[r]^-{\mathcal{K}}  & \mathcal D(\Delta_D) \ar[r]^-{-\alpha\Delta_D} & 
            L^2(\omega_p) \ar[r]^-{(D\mathcal{E})^{-1}} & \mathcal D(\mathcal E) \ar[r]^-{\Delta_D} & \mathcal D (\Delta_D) \ar[r]^-{-\alpha\tilde{\mathcal{K}}~}  & U.
        }
    \end{equation}
That $B \in \mathscr L(U)$ follows after identifying that $\Delta_D\mathcal{E}^{-1}\Delta_D \in \mathscr{L}(L^2(\omega_p))$ is the identity map ($\mathcal E =[-\Delta_D]^2$) on $H^s(\omega_p)$. Additionally, we note that if $w \in \mathcal D(\Delta_D)$ then $\gamma_0[w]=0$ and hence $x_3 w =0$ when restricted to the lateral boundary $\Gamma_c \times (-h,h)$ of $\Gamma_p$; thus $[\mathcal K \tilde{\mathcal K} w]\big|_{\Gamma_c}=0$. \end{remark}

    \begin{remark} We point out that the above fact, $\Delta_D\mathcal{E}^{-1}\Delta_D \in \mathscr{L}(L^2(\omega_p))$, is critical in the use of the abstract theory of weak solutions. For other possible configurations, this composition need not be the identity, however, it is clear that compatibility of domains (in the sense of elliptic theory) is paramount.\end{remark}
    
  We conclude with two important properties of the $B$ operator in the context of what is to follow.
\begin{proposition}\label{prop2} The operator
$B: L^2(\Omega_p) \to L^2(\Omega_p)$ is self-adjoint and monotone. Moreover, $c_p\mathbf I+B ~:L^2(\Omega_p) \to L^2(\Omega_p)$ is an isomorphism. 
\end{proposition}
\begin{proof}
By the adjoint relationship between $\mathcal K$ and $\tilde{\mathcal K}$ on $L^2(\Omega_p)$, it is clear that for $p,q \in L^2(\Omega_p)$ we have
$$(Bp,q)_{\Omega_p}=(\beta\tilde{\mathcal K}\mathcal K p, q)_{\Omega_p} = \beta(\mathcal Kp, \mathcal Kq)_{\omega_p}=(p,\beta\tilde{\mathcal K}\mathcal Kq)_{\Omega_p},$$ as well as 
\begin{equation}\label{BNorm}(Bp,p)_{\Omega_p} = \beta ||\mathcal Kp||^2_{L^2(\omega_p)}\ge 0.\end{equation}
That $c_p\mathbf I+B$ is an isomorphism on $L^2(\Omega_p)$ follows immediately from the Lax-Milgram Theorem \cite{kesavan,ciarletbook} applied to to the bilinear form
~~$b(p,q)=\left([c_p\mathbf I+B]p,q\right)=\left(p,[c_p\mathbf I+B]q\right)$ on $L^2(\Omega_p)$.
\end{proof}
\subsection{Translation to Eliminate Momentum Source}\label{Translation}

    As was done in~\cite{bw}, we provide a translation of the the system (\ref{QSPlate*}) in order to consider a simplified problem with $f \equiv 0$. This will allow us to directly frame the system in the implicit setting (\ref{gensys}). We will perform our main analysis on the translated problem. Subsequently, we will refer to this section and translate back to the original system.
    Consider the abstract problem:
    \begin{equation}\label{QSTranslation}
    \begin{cases}
        D\mathcal{E}(w) = -\alpha \Delta_D \mathcal{K}(p)+f &\in L^2(0,T;W'),\\
        [c_pp - \alpha\tilde{\mathcal{K}}\Delta_Dw]_t - A(t)p = g &\in L^2(0,T;V'),\\
        [c_pp - \alpha\tilde{\mathcal{K}}\Delta_Dw](\vec{x},0) = d_0(\vec{x}) &\in L^2(\Omega_p).
    \end{cases}
    \end{equation}
    Since $f \in H^1(0,T;W')$ by the (critical) regularity hypothesis of Theorem \ref{main2}, we can simply define
    \[
        w_f(t) = D^{-1}\mathcal{E}^{-1}f(t) \in H^1(0,T;W).
    \]
 Then, considering the variable $u = w - w_f$, we note that $w$ solves \eqref{QSTranslation} if and only if $u$ solves
    \begin{equation}\label{QSTranslation2}
    \begin{cases}
        D\mathcal{E}(u) = -\alpha \Delta_D \mathcal{K}(p) &\in L^2(0,T;W'),\\
        [c_pp - \alpha\tilde{\mathcal{K}}\Delta_Du]_t - A(t)p = g + \alpha\tilde{\mathcal{K}}\Delta_Dw_{f,t} &\in L^2(0,T;V'),\\
        [c_pp - \alpha\tilde{\mathcal{K}}\Delta_Du](\vec{x},0) = d_0(\vec{x}) + \alpha\tilde{\mathcal{K}}\Delta_D w_f(\vec{x},0) &\in L^2(\Omega_p).
    \end{cases}
    \end{equation}
    Thus, by re-scaling $g$ and $d_0$, we obtain an equivalent problem with $f \equiv 0$. Note that the temporal regularity of $f \in H^1(0,T;W')$ permits taking the time trace of $\Delta_Dw_f\big|_{t=0} \in L^2(\omega_p)$.
    
\section{Poro-elastic Plate Solutions: Proof of Theorem \ref{main2}}
First, we must place the dynamics presented in Section \ref{modelsec} into the abstract form of Section \ref{abstractsec}. After this, we invoke Theorem \ref{ImplicitExistence}, which yields a weak solution satisfying the stability estimate;  uniqueness will be obtained immediately after showing Assumption \ref{KAssumption2} provides the ``regularity" criterion from Theorem \ref{ImplicitUniqueness}. Finally, we show that the weak solution in the sense of Definition \ref{weaksol} can be recovered from the abstract weak solution, yielding Theorem \ref{main2}. 

\begin{proof}
First, we note that the system in \eqref{QSPlate*} can be abstractly presented as \eqref{QSTranslation}, using the notation of Section \ref{Operators}. Under the hypothesis on the data $f \in H^1(0,T;W')$, we invoke the translation in Section \ref{Translation} and rescale the data $g \in L^2(0,T;V')$ and $d_0\in L^2(\Omega_p)$ to obtain an equivalent system with $f\equiv 0$.
Now, suppose the permeability function $k_p$ satisfies the hypotheses of Assumption \ref{KAssumption}. We will obtain the existence of a weak solution $(w,p) \in L^2(0,T;W) \times L^2(0,T;V)$ to (\ref{QSTranslation2}). That is, we interpret the problem through the bilinear forms given in Section \ref{Operators} and take the time derivative distributionally in $\mathscr{D}'(0,T)$ yielding an equivalent form  Definition \ref{weaksol}.

With the operators as defined in Section \ref{Operators}, we reformulate (\ref{QSTranslation}) as the equivalent implicit Cauchy problem: Find $p \in L^2(0,T;V)$ such that 
\begin{equation}\label{QSCauchy}
\begin{cases}
    [(c_p \vec{I} + B)p]_t + A(t)p = g &\in L^2(0,T; V'),\\
    [(c_p \vec{I} + B)p](0) = d_0, &\in L^2(\Omega_p).
\end{cases}
\end{equation}
Note that since $[c_p \vec{I} + B]$ is an isomorphism on $L^2(\Omega_p)$, specifying $d_0 \in L^2(\Omega_p)$ is equivalent to specifying $p(0) \in L^2(\Omega_p)$ (cf. Remark \ref{ICremark}). Moreover, there exist constants such that \begin{equation}\label{data} c||d_0||_{L^2(\Omega_p)} \le ||p(0)||_{L^2(\Omega_p)} \le C||d_0||_{L^2(\Omega_p)}.\end{equation}
Since $k_p$ satisfies the hypotheses of Assumption \ref{KAssumption}, namely $0 < k_\ast \leq k_p(\vec{x},t)$ for all $\vec{x} \in \Omega_p$ and $t \in [0,T]$, $A(t)$ is strictly monotone with 
\begin{equation}\label{ANorm}
    A(p,p;t) = \| \sqrt{k_p}\partial_3 p \|_{\Omega_p}^2 \geq k_\ast\|\partial_3 p\|_{\Omega_p}^2
\end{equation}
for every $p \in V$, $t \in [0,T]$. 

From Propositions \ref{prop1} and \ref{prop2}, we have that $B \in \mathscr L(L^2(\Omega_p))$ is monotone and self-adjoint.
Moreover, (\ref{BNorm}) and (\ref{ANorm}) yield that for all $p \in V$ and $t \in [0,T]$
\begin{align}\label{QSCoercive}
    ([c_p\mathbf I + B]p,p) + 2A(p,p;t) &\geq c_p\|p\|_{\Omega_p}^2 + \beta\|\mathcal{K}(p)\|_{\omega_p}^2 + 2k_\ast\|\partial_3p\|_{\Omega_p} \nonumber\\
    &\geq \min(c_p,k_\ast)\|p\|_V^2.
\end{align}
Therefore, we are in a position to invoke Theorem \ref{ImplicitExistence}, which guarantees the existences of a weak solution $p \in L^2(0,T;V)$ to \ref{QSCauchy} satisfying the stability estimate
\begin{equation}\label{preest}    \|p\|_{L^2(0,T;V)}^2 \leq C(c_p,k_\ast)\left[\|g\|^2_{L^2(0,T;V')} + (d_0,p(0))\right].\end{equation}

To show that this weak solution is indeed unique, we will further require that permeability $k_p$ satisfies Assumption \ref{KAssumption2}. That is that $k_p(\vec{x},\cdot)$ is absolutely continuous in $t$ for a.e. $\mathbf x \in \Omega_p$, and $|\partial_t k_p(\vec{x},t)| \leq K(t)$ with $K \in L^1(0,T)$. Then, straightforwardly,
\begin{align*}
    \left| \frac{d}{dt}A(p,q;t) \right| 
    \leq  \int_{\Omega_p} |\partial_tk_p(\vec{x},t)|\,|\partial_3p|\,|\partial_3q|\,d\vec{x} 
    \leq K(t)\|p\|_V\|q\|_V,\quad \forall\, p,q \in V, ~ \forall\,t \in [0,T].
\end{align*}
Thus, $\big\{A(t): V\to V' : t \in [0,T]\big\}$ is a regular family. As we have noted in Section \ref{constOp}, $A(t)$ is self-adjoint on $L^2(\Omega_p)$ (also in the $V \to V'$ sense \cite{showmono}). Thus, Theorem \ref{ImplicitUniqueness} holds, and $p \in L^2(0,T;V)$ obtained in the previous step is indeed the unique weak solution to the Cauchy problem \eqref{QSCauchy}.  

Finally, we must connect our weak solution of \eqref{QSCauchy} back to a weak solution in the sense of Definition \ref{weaksol}.
With the unique weak solution $p \in L^2(0,T;V)$ to the problem \eqref{QSCauchy}, we have 
\[
    -\frac{\alpha}{D}\Delta_D\mathcal{K}(p) \in L^2(0,T;V').
\]
Then by the elliptic solver $\mathcal{E}^{-1} : W' \to W$, we obtain a unique weak solution $w \in L^2(0,T;W)$ to the  problem $\mathcal Ew=-\alpha \Delta_D\mathcal K p$. We then make the identification that $Bp=-\alpha\tilde{\mathcal K}\Delta_Dw$. Hence, a weak solution \eqref{QSCauchy} yields a (distributional-in-time) solution to \eqref{QSTranslation2}.
Now, we undo the translation of the variable $w \mapsto w - w_f$, in view of Section \ref{Translation}. This yields a unique solution to \eqref{QSTranslation}  which we also denote by $(w,p) \in L^2(0,T;W\times V)$. Connecting the obtained solution of the abstract equation \eqref{QSTranslation} back to Definition \eqref{weaksol} follows immediately from the equivalence of Definition \ref{weakimplicitsol} and \eqref{ImplicitCauchy}, and identifying the operators $\cE, \mathcal K,$ and $\tilde{\mathcal K}$ from Section \ref{Operators}.  To obtain the stability estimate as presented in Theorem \ref{main2}, we begin with the one obtained for $p$ in the previous step \eqref{preest} and undo the translation to obtain:
{\small \begin{equation} \label{stab2}   \|p\|_{L^2(0,T;V)}^2 \leq C(c_p,k_\ast)\left[\|g\|^2_{L^2(0,T;V')}+\|\alpha\tilde{\mathcal{K}}\Delta_D w_f\|^2_{L^2(0,T;V')} + (d_0,p(0))+\left(\alpha\tilde{\mathcal{K}}\Delta_D w_f(0),p(0)\right)\right]. \end{equation}}
We then collect the following inequalities for the weak solution\footnote{allowing $C$ to be a general constant that may change from line to line}:
{\small \begin{align*}
||w||_{W} \le& ~C\left(\frac{\alpha}{D}\right)\left|\left|\mathcal K p \right|\right|_{L^2(\omega_p)}+||\Delta_D^{-1}f||_{L^2(\omega_p)}
\le ~C\left[||p||_{L^2(\Omega_p)}^2+||f||_{W'}\right], \\
||p(0)||_{L^2(\Omega_p)} \le&~ C||[c_0\mathbf I+B]p(0)||_{L^2(\Omega_p)} \le C||d_0||_{L^2(\Omega_p)},\\
||\alpha\tilde{\mathcal{K}}\Delta_D w_f||_{V'} \le & ~\frac{\alpha}{D}||\tilde{\mathcal K}\Delta_D\mathcal E^{-1}f||_{V'} \le C||\tilde{\mathcal K}\Delta_D^{-1}f||_{L^2(\Omega_p)} \le C||\Delta_D^{-1}f||_{L^2(\omega_p)} \le C||f||_{W'},\\
||F\big|_{t=0}||_{W'} \le&~ C ||F||_{H^1(0,T;W')}.
\end{align*}}
Using these, the stability estimate for $p$ in \eqref{stab2} becomes the estimate in \eqref{stability**} of Theorem \ref{main2}.\end{proof}
	
\section{Future Work and Open Problems}\label{future}

	\subsection{Incompressible Constituents} As in the analysis of the 3D Biot dynamics, for instance in \cite{cao2,bw,bcmw}, the presence of compressibility $c_p>0$ has a clear impact on analytical approaches for the Biot plate. In the approach taken here, when $c_p=0$, we lose control of the $L^2(\Omega_p)$ norm in the static analysis, and hence coercivity of the reduced problem on $V$ (see \eqref{QSCoercive}). Moreover, the incompressible ``fluid content operator" $B=0\mathbf I +B$ is  non-injective on $L^2(\Omega_p)$, disqualifying many steps in the above analysis. 
	
	On the other hand, the recent work in \cite{bw} utilizes a singular limit approach to obtain existence in the case of vanishing compressibility.  For such an approach to obtain here, a variety of adjustments should be made to the functional setting (e.g., incorporating the lateral boundary conditions), making this issue not unrelated to the next  point in Section \ref{configs}. If existence can be obtained, a recent mollification argument presented in \cite{borisnew} for the 3D incompressible Biot system  holds promise to obtain energy estimates for the  incompressible case {\em without assuming} that the permeability $k_p(t)$ is {\em regular} (that is, without bounds on the first derivative). Hence, one might obtain the analogous result for well-posedness of weak solutions with $c_p=0$ and only require Assumption \ref{KAssumption}, dispensing with Assumption \ref{KAssumption2}, to obtain both continuous dependence and uniqueness. 
		
		\subsection{Other Configurations of Interest} \label{configs}
	In the analysis herein, it is critical to have comparability of the realization of $\Delta$, applied to the pressure moment, and that of $\Delta^2$, associated to elasticity, to make use of the abstract setup for weak solutions in Section \ref{abstractsec}. In general, when taken with boundary conditions reflecting different physical configurations, these operators may be defined on non-comparable domains. Any configuration other than the one considered  here  will require a substantial reworking of the treatment of the $B$ operator, and the general abstract setup. In a subsequent work, we aim to produce the operator-theoretic framework and associated analysis for the Biot plate problem taken with clamped plate conditions, along with other approaches to conditions for conditions on the lateral sides.
	
		\subsection{Nonlinearities} As motivated by biological applications (see, e.g., \cite{bgsw,bcmw}), it is quite natural to consider the permeability function with the dependence $k=k({\zeta})$. With the constitutive law in place here, we would then have $k_p=k_p(c_pp-\alpha x_3\Delta w)$, rendering the problem quasilinear. While it might seem natural to utilize a fixed point scheme in the context of our time-dependent results here, even in the 3D case this is technically challenging \cite{bw,borisnew}. On the other hand, there is a critical difference between the 3D Biot model and the plate here, the primary difference occuring in the elliptic degeneracy associated to the pressure equation; namely, there is no clear in-plane smoothing for $p$ or $w$.  This presents a substantial challenge in utilizing strategies such as those employed in \cite{cao2,bw,bgsw}.
		
		Alternatively, for plate problems arising in application, it is always of interest to consider large deflections. For traditional plate theory, this is the refinement of Kirchhoff-Love theory to that of von Karman's  \cite{ciarlet2,yellowbook}. Yet again, it will not be immediate to add semilinear elastic terms to the $w$-equation, since weak solutions for von Karman's equations are notoriously troublesome. Additionally, obtaining {\em smooth} solutions to the Biot plate problem is, at present, open (see the following Subsection \ref{smooth}). Moreover, if one retains the in-plane dynamics from the Biot plate model in \cite{mikelic}, it would be natural to consider the so called full von Karman equations of elasticity \cite{ciarlet2}; these equations employ coupled nonlinear terms in both the transverse dynamics as well as the in-plane dynamics. In either case, the incorporation of nonlinearity in the plate presents formidable challenges.

	\subsection{Implicit Degenerate Semigroup Theory}\label{smooth} In the analysis presented in the body of this work we have addressed only weak solutions, that, by their definition are indeed quite weak. No discussion has been made of strong solutions, or, relatedly, regularity theory for the quasi-static Biot plate system. (Though, some of the issues associated to lack of smoothing and boundary conditions for smooth solutions can be inferred from the semigroup discussion in Appendix B.) Yet, as we have seen in the various sections above, there is certainly motivation for a regularity theory to the Biot plate system. The most direct path would likely utilize the implicit semigroup approach presented (as applied to linear Biot systems) in \cite{show1}, with substantial generalizations in \cite{dbshow}. Though we must point out that, for technical reasons, these approaches are largely in the $V\to V'$ context and may not directly permit the requisite energy estimates for working with fixed point approaches to nonlinear problems. Moreover, these implicit semigroup approaches do not admit time-dependent operators, and, as this is a focus here, would require adaptation to a ``evolution" operator framework \cite{pazy}. However, a theory of time-dependent evolution operators associated to implicit equations does not seem readily available at present.

\section{Appendices}

\section*{Appendix A: Proof of Theorems \ref{ImplicitExistence}--\ref{ImplicitUniqueness}}\label{appendixa}
Consider the setting of Section \ref{abstractsec}. 
The following proofs track those given in \cite[III.3]{showmono}.

\subsection{Proof of Theorem \ref{ImplicitExistence}}
Suppose that $\cB$ is self-adjoint and monotone on $H$. Additionally, suppose there exist constants $\lambda,c > 0$ such that
\[
    2\langle \cA(t)v,v\rangle + \lambda(\cB v,v) \geq c\|v\|_V^2,\quad \forall\, v \in V,~\forall\, t \in [0,T].
\]
To prove existence of a weak solution to the implicit Cauchy problem, we will need the following extension of Lax-Milgram due to Lions, as well as a subsequent corollary. 
\begin{theorem}[Lions]\label{Lions}
    Let $(\BA, \|\cdot\|_\BA)$ be a Hilbert space and $(\BB, \|\cdot\|_\BB)$ be a normed linear space. If $a : \BA \times \BB \to \mathbb{R}$ is a bilinear form such that $a(\cdot,\phi) \in \BA'$ for every $\phi \in \BB$, then TFAE:
    \begin{itemize}
        \item $\displaystyle{\inf_{\|\phi\|_\BB = 1} \sup_{\|u\|_{\BA} \leq 1} |a(u, \phi)| \geq c > 0}$,
        
        \item for each $F \in \BB'$, there exists a $u \in \BA$ such that:~ $a(u,\phi) = F(\phi)$ for all $\phi \in \BB$.
    \end{itemize}
\end{theorem}
\begin{corollary}\label{LionsCor}
    If $\BB$ is continuously embedded in $\BA$ and $a$ is $\BB$-elliptic, then Theorem \ref{Lions} holds. 
\end{corollary}

\begin{proof}[Proof of Theorem \ref{ImplicitExistence}]
Set $\BA = L^2(0,T;V)$ with the usual norm and $\BB = \{\phi \in \BA ~:~ \phi' \in L^2(0,T;H), \phi(T) = 0\}$ with the norm given by $\|\phi\|_\BB^2 = \|\phi\|_\BA^2 + (\cB\phi(0),\phi(0))_H$. This norm is well-defined by the monotonicity of $\cB$. For $u \in \BA$ and $\phi \in \BB$, let
\begin{align*}
    a(u,\phi) &= \int_0^T\langle \cA(t)u(t), \phi(t)\rangle \,dt - \int_0^T (\cB u(t),\phi'(t))\, dt;~~~~~F(\phi) = \int_0^T \langle S(t),\phi(t)\rangle dt + (\cB u_0,\phi(0)).
\end{align*}
Then the weak formulation of the implicit, degenerate Cauchy problem \eqref{ImplicitCauchy} given in \eqref{weakweak} is equivalent to the problem: 
~~Find $u \in \BB$ such that $a(u,\phi) = F(\phi)$ for all $\phi \in \BB$.

Note that $u \in \BA$ if and only if the exponential shift $v$ of $u$ is. That is, $u \in \BA$ if and only if $v \in \BA$ with $v(t) \equiv e^{-\lambda t}u(t)$ for any $\lambda \in \mathbb{R}$. Then $u$ solves the implicit Cauchy problem (\ref{ImplicitCauchy}) exactly when $v$ is the solution of the problem: ~~Find $v \in \BA$ such that 
\[
    \begin{cases}
        \frac{d}{dt}[\cB v(t)] + [\lambda I + \cA(t)]v(t)= e^{-\lambda t}S(t),\\
        \cB v(0) = \cB u_0.
    \end{cases}
\]
So choose $\lambda$, so that by the exponential shift, the estimate
\[
    2\langle \cA(t)v,v\rangle \geq c\|v\|_V^2,\quad \forall v\in V, ~ t\in [0,T]
\]
is equivalent to the supposed estimate. Since, by the regularity of $\phi \in \BB$, we have 
$$\frac{d}{dt}(\cB\phi(t),\phi(t)) = (\cB\phi'(t),\phi(t)) + (\cB\phi(t),\phi'(t)),$$ it follows that
\begin{align*}
    2a(\phi,\phi) = & \int_0^T 2\langle \cA(t)u(t),\phi(t)\rangle\, dt -\int_0^T (\cB \phi(t),\phi'(t))\,dt\\
    &+ \int_0^T (\cB\phi'(t),\phi(t))_H\,dt - \int_0^T \frac{d}{dt}(\cB\phi(t),\phi(t))\,dt,\quad \forall\, \phi \in \BB.
\end{align*}
Since $\cB$ is assumed self-adjoint and $\phi \in \BB$ (so $\phi(T) = 0$), this simplifies to
\[
    2a(\phi,\phi) \geq c\|\phi\|_\BA^2 + (\cB\phi(0),\phi(0)) \geq \min(1,c) \|\phi\|_\BB^2,\quad \forall\, \phi \in \BB.
\]
That is, the bilinear form $a(\cdot,\cdot)$ is $\BB$-elliptic. Since we also have by the Aubin-Lions Lemma \cite[Propositions 1.2 and 1.3]{showmono} that $\BB \hookrightarrow \BA$ continuously, it follows from Corollary \ref{LionsCor} and Theorem \ref{Lions} that there exists weak solution $u \in L^2(0,T;V)$ to the Cauchy problem described in (\ref{weakweak}). 
\end{proof}
\subsection{Proof of Theorem \ref{ImplicitUniqueness}}
In addition to the assumptions taken in the previous section, let us also assume that $\{\cA(t): V \to V~:~[0,T]\}$ is a regular family of self-adjoint operators. Again, we have assumed there are $\lambda,c >0$ such that
\[
    \langle \cA(t)v, v\rangle + \lambda(\cB v,v) \geq c\|v\|^2_V, \quad \forall\, v \in V, ~ \forall\, t \in [0,T].
\]
As in the previous proof, we invoke the exponential shift, so the above is equivalent to
\[
    \langle \cA(t)v,v\rangle \geq c\|v\|_V^2, \quad \forall\, v \in V, ~ \forall\, t \in [0,T].
\]
Let $a$, $\BA$, and $\BB$ be the bilinear form and Sobolev spaces defined in the preceding proof. We will need the result below from \cite{showmono}.
\begin{proposition}\label{RegularProp}
    If $\{\cA(t)\}$ is a regular family of operators in $\mathscr{L}(V,V')$, then for every $u,v \in H^1(0,T;V)$, $\langle \cA(\cdot)u(\cdot),v(\cdot)\rangle $ is absolutely continuous on $[0,T]$ and 
    \[
        \frac{d}{dt} \langle \cA(t)u(t),v(t) \rangle= \langle \cA'(t)u(t),v(t)\rangle + \langle \cA(t)u'(t), v(t)\rangle + \langle \cA(t)u(t),v'(t)\rangle, \quad \text{a.e. }\, t \in [0,T].
    \]
\end{proposition}
\noindent Now, we suppose that $u$ is a weak solution to the Cauchy problem described in (\ref{weakweak}) with $u_0 = 0, F = 0$. Since $u$ is not necessarily in $\BB$, we construct a test function in the following way. Let $s \in (0,T)$ and define $v\in H^1(0,T;V)$ by $v(t) = -\int_t^su(\tau)\,d\tau$ for $t \in [0,s]$ and $v(t) = 0$ for $t \in [s,T]$. Then $v'(t) = u(t)$ for $t \in (0,T)$ and $v(t) = v'(t) = 0$ for $t \in (s,T]$. Hence, $v \in \BB$. It follows from Theorem \ref{Lions} that
\[
    a(u,v) = \int_0^s \langle \cA(t)v'(t), v(t)\rangle \,dt - \int_0^s (\cB u(t),u(t))\,dt = 0.
\]
Since the family $\{A(t)\}$ is regular, Proposition (\ref{RegularProp}) guarantees that
\begin{align*}
    -2a(u,v) = & \int_0^s 2(\cB u(t),u(t))\,dt - \int_0^s \langle \cA(t)v'(t),v(t)\rangle\,dt + \int_0^s \langle \cA'(t)v(t),v(t)\rangle\,dt\\
    & + \int_0^s \langle \cA(t)v(t), v'(t)\rangle \,dt - \int_0^s \frac{d}{dt} \langle \cA(t)v(t),v(t)\rangle.
\end{align*}
Since each $\cA(t)$ is self-adjoint and $v(s) = 0$, we have that 
\[
    \int_0^s 2(\cB u(t),u(t))\,dt + \int_0^s \langle \cA'(t)v(t), v(t)\rangle\,dt + \langle \cA(0)v(0),v(0)\rangle = 0.
\]
Now, set $U(t) = \int_0^t u(\tau)\,d\tau$ so that $U(s) = -v(0)$ and $U(t) - U(s) = v(t)$ for $t \in (0,s)$. By the assumed coercivity, and since $\cB$ is monotone,
\begin{align*}
    c\|U(s)\|_V^2 &\leq \langle \cA(0)U(s),U(s)\rangle + \int_0^s 2(\cB u(t),u(t))\,dt \leq -\int_0^s \langle \cA'(t)v(t),v(t)\rangle \,dt.
\end{align*}

Since the family $\{\cA(t)\}$ is regular, there is a $K \in L^1(0,T)$ such that
\[
    c\|U(s)\|_V^2 \leq \int_0^s K(t)\|v\|_V^2\,dt \leq 2\int_0^s K(t)\left(\|U(t)\|_V^2 + \|U(s)\|_V^2 \right)\,dt.
\]
We can choose $s_0$ sufficiently small that $k = 2\int_0^{s_0}K(t)\,dt < c$. Then
\[
    \|U(s)\|_V^2 \leq \frac{2}{c-k} \int_0^s K(t)\|U(t)\|_V^2\,dt,\quad \forall\, s \in [0,s_0].
\]
Hence by the Gr\"onwall inequality, $U(s) = 0$ for $s \in [0,s_0]$ and subsequently $u(t) = 0$ for $t \in (0,s_0)$. Since $K \in L^1(0,T)$, choose $s_0$ so that $\int_\tau^{\tau + s} K(t)\,dt < c/2$ for $\tau \in [0,\tau-s_0]$. Then, apply the preceding a finite number of times to obtain $u(t) = 0$ for $t \in [0,T]$. 

\section*{Appendix B: Inertial, Compressible Case}\label{appendixb}
In this appendix we address the inertial Biot plate \eqref{FDPlate}, specifically in the case $\rho_p,c_p>0$. We utilize the semigroup theory, motivated by classical references for coupled parabolic systems (such as that of thermoelasticity) as in \cite{thermo,redbook}. We have chosen to keep this appendix self-contained, owing to its thematic departure from the main body of the manuscript. In this appendix, we will frame the problem abstractly, largely utilizing the operators and spaces from Section \ref{Operators}. Moreover, for clarity of exposition and brevity, we consider the case $k_p=const>0$ (though we comment on the time-dependent case below). In this framework, we will demonstrate that the evolution operator associated to this parabolic-hyperbolic system generates a strongly continuous semigroup on the appropriate state space. With that result in hand, we can invoke the traditional procedure of obtaining weak solutions from mild solutions, identified as $C_0$ strong limits of strong solutions (that is, solutions for data in the domain of the generator); see \cite[pp.258--259]{pazy} and \cite[Section 2.4, pp.75--90]{yellowbook}. Finally, we will remark on the extension to the time-dependent case with $k_p=k_p(t)$ utilizing the notion of a stable family of generators \cite[Ch. 5.3]{pazy}. 

Consider the inertial Biot plate as before with $c_p,\rho_p>0$, after re-scaling:
\begin{equation}\label{FDPlate-B} \small
\begin{cases}
    w_t-v = 0 & \text{in } \omega_p\\
    v_t + {D} \Delta^2 w + {\alpha}_p \Delta \int_{-h}^h x_3p\,dx_3 = f &\text{in } \omega_p,\\
    \partial_tp - {\alpha}_p x_3\Delta v - \partial_3(k_p\partial_3 p) = g & \text{in } \Omega_p,\\
    w(x_1,x_2,0)=w_0(x_1,x_2),~~w_t(x_1,x_2,0)=w_1(x_1,x_2) & \text{in } \omega_p,\\
    p(\mathbf x,0)-\alpha x_3\Delta w_1(x_1,x_2)=d_0(\mathbf x) & \text{in } \Omega_p,\\
        w  = 0;~~D\Delta w+\alpha\int_{-h}^hx_3 p~dx_3=0 &\text{on } \Gamma_c,\\
        \partial_{\vec{n}} p = 0 &\text{on } \{x_3=h\} \cup \{x_3=-h\}.
    \end{cases}
\end{equation}
The primary result  will correspond to the case of constant permeability $k_p=const > 0$. Hence we replace the definition of the operator $A(t)$ with $A=-k\partial_3^2$ defined on 
\[
    \mathcal D(A) = \set{u \in H^{0,0,2}(\Omega_p) ~: \gamma_1[u] = 0 \text{ on } \{x_3 = \pm h\}}.
\]

We utilize the spaces in Section \ref{spaces}, and recall
\begin{align}
\mathcal D(\mathcal E) = & \{w \in H^4(\omega_p)\cap H_0^1(\omega_p)~:~\gamma_0[\Delta w]=0\}=\mathcal D(\Delta_D^2),\\
\mathcal D(\mathcal E^{1/2})  = & ~W = H^2(\omega_p) \cap H_0^1(\omega_p);~~~~~V=~H^{0,0,1}(\Omega_p).
\end{align}

\subsection{Semigroup Generation}	
\begin{theorem}
    Consider the inertial plate system (\ref{FDPlate-B}) with $f=g=0$, in $\mathbf y =[w,v,p]$, posed on 
    \[
        X = \mathcal{D}(\mathcal{E}^{1/2}) \times L^2(\omega_p) \times L^2(\Omega_p)
    \]
    endowed with the norm $\norm{\vec{y}}_{{X}}^2 = \norm{\mathcal{E}^{1/2} w}_{0,\omega_p}^2 + \norm{v}_{0,\omega_p}^2 + \norm{p}_{0,\Omega_p}^2$. Define the matrix operator 
    \[
        \mathbf{A}: \mathcal{D}(\mathbf{A}) = \Big\{[w,v,p] \in  W \times W \times  \mathcal D(A)~:~[\Delta w+\alpha\mathcal Kp] \in \mathcal D(\Delta_D)\Big\}
    \]
    with differential action
    \[
        \mathbf A\vec{y}  
        = \begin{pmatrix} v\\ -\mathcal E^{1/2}\left[\mathcal{E}^{1/2} w + \alpha \mathcal K p\right]\\[.2cm] \alpha \tilde{\mathcal K}\Delta_D v - A p \end{pmatrix},\quad \forall\, \vec{y} \in \mathcal{D}(\mathbf A).
    \]
    Then $\mathbf A$ is the generator of a strongly continuous semigroup $\{e^{\mathcal{\mathbf A}t}~:~ t \geq 0\}$ of contractions on X corresponding to the Cauchy problem
    $$\dot{\mathbf y} = \mathbf A\mathbf y;~~ \mathbf y(0)=[w_0,w_1,p] \in \mathcal D(\mathbf A).$$
\end{theorem}

\begin{proof}
Note that 
{
\begin{align} (\mathbf A\mathbf y, \mathbf y)_X=&~ (\mathcal E^{1/2}v,\mathcal E^{1/2}w)_{\omega_p}-\big(\mathcal E^{1/2}[\mathcal E^{1/2}w+\alpha \mathcal Kp],v\big)_{\omega_p}+(\alpha\tilde{\mathcal K}\mathcal E^{1/2}v-Ap,p)_{\Omega_p} \\
=&~(\mathcal E^{1/2}v,\mathcal E^{1/2}w)_{\omega_p}-\big(\mathcal E^{1/2}w,\mathcal E^{1/2}v\big)_{\omega_p}-(\alpha\mathcal K p, \mathcal E^{1/2}v)_{\omega_p}+(\alpha\tilde{\mathcal K}\mathcal E^{1/2}v,p)_{\Omega_p}\\\nonumber
&-(Ap,p)_{\Omega_p} \\
=&-(k_p\partial_3p,\partial_3p)_{\Omega_p} \le 0,~ \quad \forall\, \vec{y} = [w,v,p] \in \mathcal D(\mathbf A).
\end{align}}
We utilized the self-adjointness of $\mathcal E^{1/2}$ on $L^2(\omega_p)$ and the adjoint relation between $\mathcal K$ and $\tilde{\mathcal K}$. 

To show that $\mathbf A$ is indeed $m$-dissipative, we consider for a given $\vec{F} = [f_1,f_2,f_3] \in X$ the resolvent system $(\vec{I} - \mathbf A)\vec{y} = \vec{F}$ for $\vec{y} = [w,v,p] \in \mathcal{D}(\mathbf A)$, i.e.,
    \begin{equation}\label{FDResolvent1}
    \begin{cases}
        w - v = f_1 \in W, \\
            v + \mathcal{E}^{1/2}[\mathcal E^{1/2}w +\alpha \mathcal Kp] = f_2 \in L^2(\omega_p), \\
            p -\alpha \tilde{\mathcal K}\mathcal E^{1/2}v + Ap = f_3 \in L^2(\Omega_p).
    \end{cases}
    \end{equation}
    By the first equation, $v = w - f_1$, whence
    \begin{equation} \label{FDResolvent2}
    \begin{cases}
         w +\mathcal{E}^{1/2}[\mathcal E^{1/2}w +\alpha \mathcal Kp]= f_2 + f_1 \in L^2(\omega_p), \\
      p -\alpha \tilde{\mathcal K}\mathcal E^{1/2}w + Ap = f_3 -\alpha \tilde{\mathcal K}\mathcal E^{1/2}f_1 \in L^2(\Omega_p).
    \end{cases}
    \end{equation}
    Note that these equations are self-consistent on $L^2(\omega_p)\times L^2(\Omega_p)$, since $f_1 \in W = \mathcal{D}(\mathcal E^{1/2})$, and 
 $p \in \mathcal D(A)$ and $[\mathcal E^{1/2}w+\alpha \mathcal Kp] \in  W$. Now,  for  (\ref{FDResolvent2}), we can generate a bilinear form for $\phi=[w,p]$ from the LHS:
    \begin{align}\label{FDBilinear}
        a(\phi,\psi) := &~D\int_{\omega_p} \Delta_D w \Delta_D u\,dx_1dx_2 + \int_{\omega_p} wu\,dx_1dx_2 + \int_{\omega_p} \alpha {\mathcal{K}}p [\Delta u] \,d\vec{x}\\
        &~- \int_{\omega_p}  \alpha  {\mathcal{K}}q [\Delta w] \,d\vec{x} + \int_{\Omega_p} pq\,d\vec{x} + \int_{\Omega_p} k_p \partial_3p\partial_3q\,d\vec{x}, \quad \forall\, \psi = [u,q] \in W \times V. \nonumber
    \end{align}
     Then we can state a weak formulation of the resolvent equation:~~ Find $\phi \in W \times V$ such that 
    \begin{equation}\label{FDweak}
        a(\phi,\psi) = \int_{\omega_p}(f_2u + f_1u)\,dx_1dx_2 + \int_{\Omega_p}(f_3 -  \alpha \tilde{\mathcal{K}}\Delta_D f_1)q\,d\vec{x}, \quad \forall\, \psi = [u,q] \in W \times V.
    \end{equation}
    
    Under the $||\Delta\cdot ||_{L^2(\omega_p)}$ norm on $W$ (equivalent thereon to the full $||\cdot||_{H^2(\omega_p)}$ norm), coercivity of the bilinear form $a(\cdot,\cdot)$ is evident:
    \begin{align*}
        a(\phi,\phi) &= D\|\Delta_D w\|_{0,\omega_p}^2 + \|w\|_{0,\omega_p}^2 + \|p\|_{0,\Omega_p}^2 + k_p\|\partial_3 p\|_{0,\Omega_p}^2\\
        &\geq D\norm{w}_{W}^2 + \min(1,k_p)\norm{p}_V^2 \quad \forall\, \phi = [w,p] \in W \times V.
    \end{align*}
    Now, let $\phi = (w,p), \psi = (u,q) \in W \times V$ be arbitrary. Then by continuity of $\mathcal K,\tilde{\mathcal K}$
    
   {\small \begin{align*}
        \int_{\Omega_p} \alpha  \tilde{\mathcal{K}}(p) \Delta_D u \,d\vec{x}
        &=  \alpha \sqrt{2h}\sup\{x_3 \in [-h,h]\}\norm{ p}_{0,\Omega_p}\norm{\Delta_D  u}_{0,\omega_p}\leq \sqrt{2}\alpha h^{3/2}\norm{p}_V\norm{ u}_{2,\omega_p},\\
        \int_{\Omega_p} \alpha  \tilde{\mathcal{K}}\Delta_D w q \,d\vec{x} &=  \alpha \sup\{x_3 \in [-h,h]\}\norm{\Delta_D w}_{0,\Omega_p}\norm{q}_{0,\Omega_p}\\
        &\leq \sqrt{2}\alpha h^{3/2} \norm{w}_{2,\omega_p}\norm{q}_V.
    \end{align*}}
    It follows that 
    \begin{align*}
        a(\phi,\psi) \leq ~ &D(\Delta_D w,\Delta_D u)_{\omega_p} + (w,u)_{0,\omega_p} + ( \alpha  {\mathcal{K}}(p), \Delta_D u)_{\omega_p} - ( \alpha  \tilde{\mathcal{K}}\Delta_D w, q )_{\Omega_p} + (p,q)_{\Omega_p} + k_p(\partial_3p,\partial_3q)_{\Omega_p}\\
        \leq ~ &\max(D,1)\norm{w}_{W}\norm{ u}_{W} +  \alpha h\norm{p}_{\Omega_p}\norm{ u}_{W}  -  \alpha h\norm{ w}_{W}\norm{q}_{\Omega_p} + \max(1,k^*)\norm{p}_V\norm{q}_V\\
        \leq~& C( \alpha ,D,h,k^*)\norm{\phi}_W\norm{\psi}_V, ~~\forall~~\phi,\psi \in W \times V,\end{align*}
    which is to say that the bilinear form $a(\cdot,\cdot)$ is continuous on $W\times V$.
   By the Lax-Milgram Theorem \cite{kesavan}, we have that there is a unique solution to the variational problem \eqref{FDweak} with associated continuity estimate. 
   
   Now, considering that \eqref{FDweak} holds for all $\psi = [u,q] \in W \times V$, we choose $q=0$ to obtain for the solution $w \in W$ that
   \begin{align}
       D\int_{\omega_p} \Delta_D w \Delta_D u\,dx_1dx_2 + \int_{\omega_p} wu\,dx_1dx_2 + \int_{\omega_p} \alpha {\mathcal{K}}p [\Delta u] \,d\vec{x} = \int_{\omega_p}(f_2u + f_1u)\,dx_1dx_2
    \end{align}
for all $u \in W$. We rewrite as
      \begin{align}\label{thisequality}
       (D\Delta_D w+\alpha{\mathcal K}p,\Delta_D u)_{\omega_p}= (f_1+f_2-w,u)_{\omega_p},~~\forall u \in W.
    \end{align}
    But, since $C_0^{\infty}(\omega_p) \subseteq W$, we can choose $u$ therein to obtain from \eqref{thisequality} that 
    $$\Delta_D[\Delta_Dw+\alpha \tilde{\mathcal K}p] \in L^2(\omega_p).$$ Moreover, we may integrate by parts twice with $u \in C_0^{\infty}(\omega_p)$ to obtain
    $$\big(\Delta_D[D\Delta_D w+\alpha {\mathcal K}p],u)_{\omega_p} = (f_1+f_2-w,u)_{\omega_p},$$ and by density of the test functions in $L^2(\omega_p)$ we have the almost everywhere equality $$\Delta_D[D\Delta_D w+\alpha {\mathcal K}p]=f_1+f_2-w.$$ Moreover, if we now consider \eqref{thisequality} with $u$ taken from $W$, and undo integration by parts in the standard way, we will infer 
$$\big(\Delta_D[D\Delta_D w+\alpha {\mathcal K}p],u)_{\omega_p}+(\gamma_0[\Delta_D w+\alpha {\mathcal K}p],\gamma_1[u])_{\Gamma_c} = (f_1+f_2-w,u)_{\omega_p}.$$ But by the earlier $L^2(\omega_p)$ equality, we infer that the trace term vanishes for all $u \in W$. In the standard way, by the surjectivity of the trace mapping \cite{kesavan}, we infer that $\gamma_0[D\Delta_Dw+\alpha \mathcal K p]=0$. 

Then, with the regularity of $w$ established, the pressure equation obtains $p \in \mathcal D(A)$ as follows: With $f_3 -  \alpha \tilde{\mathcal{K}}\Delta_D f_1 +\alpha\tilde{\mathcal K}\mathcal E^{1/2}w \in L^2(\Omega_p)$, we may choose $\psi = [0,p] \in W\times V$ to obtain the variational form of
$$  p -\alpha \tilde{\mathcal K}\mathcal E^{1/2}w + Ap = f_3 -\alpha \tilde{\mathcal K}\mathcal E^{1/2}f_1 \in L^2(\Omega_p).$$
We may then apply elliptic regularity in the $x_3$-direction associated with the elliptic operator $A$.
 Finally, since we may let $v=w-f_1 \in W$, we have obtained a solution $[w,v,p] \in \mathcal D(\mathbf A)$ to the resolvent equation.
    This is now to say that the operator $I - \mathbf A : \mathcal{D}(\mathbf A) \to X$ is surjective. 
    
    By the classical form of the Lumer-Phillips theorem \cite[Section 1.4]{pazy}, $\mathbf A$ generates a strongly continuous semigroup of contractions on $X$.
\end{proof}
We note that strong or mild solutions can be obtained directly, under various (standard) hypotheses for $f$ and $g$. See for instance \cite[Section 4.2]{pazy}.

As described above, the generation result provides {\em strong solutions} for data $\mathbf y=[w,v,p] \in \mathcal D(\mathbf A)$, and {\em mild solutions} for data $\mathbf y \in X$, which are also {\em weak solutions} (in the appropriate sense). 

Moreover, under the hypotheses on $k_p(t)$ in Assumption \ref{KAssumption} we may replace $A$ with $A(t)$ ($k\mapsto k_p(t)$) and the resulting operator $\mathbf A(t)$ will induce a stable family of generators on $\mathcal D(\mathbf A(t)) \subseteq X$ satisfying the hypotheses for generation of {\em an evolution system} \cite[Ch. 5.3]{pazy}. Again, via the standard mechanisms, weak solutions can be obtained for the time-dependent problem.

\end{document}